\documentclass{amsart}
\usepackage{comment}
\usepackage{amscd}
\usepackage{amssymb}
\usepackage[all, knot]{xy}
\xyoption{all}
\xyoption{arc}
\usepackage{hyperref}

\newcommand{\edit}[1]{}

\def\Aut{\operatorname{Aut}}

\def\s{\sigma}

\def\naive{{\text{naive}}}

\def\cube#1#2#3#4#5#6#7#8{
& #5 \ar[rr] \ar[dl] \ar@{-}[d] && #6 \ar[dd] \ar[dl] \\
#1 \ar[rr] \ar[dd]  & \ar[d] & #2 \ar[dd] \\
& #7 \ar@{-}[r] \ar[dl] & \ar[r] & #8 \ar[dl] \\
#3 \ar[rr] && #4 \\
}

\def\Hom{\operatorname{Hom}}

\def\a{\alpha}
\def\b{\beta}

\def\smsh{\wedge}

\def\ker{\operatorname{ker}}
\def\cone{\operatorname{cone}}

\def\cP{\mathcal P}
\def\cE{\mathcal E}

\def\cO{\mathcal O}

\def\trace{\operatorname{trace}}

\def\Tor{\operatorname{Tor}}

\def\Spec{\operatorname{Spec}}

\def\lra{\longrightarrow}
\def\into{\hookrightarrow}
\def\onto{\twoheadrightarrow}

\newcommand{\Q}{\mathbb{Q}}

\newcommand{\C}{\mathbb{C}}
\newcommand{\Z}{\mathbb{Z}}

\newcommand{\fm}{{\mathfrak m}}

\numberwithin{equation}{section}

\theoremstyle{plain} 
\newtheorem{thm}[equation]{Theorem}

\newtheorem{introthm}{Theorem}

\newtheorem*{introthm*}{Theorem}

\newtheorem{cor}[equation]{Corollary}
\newtheorem{lem}[equation]{Lemma}
\newtheorem{prop}[equation]{Proposition}

\theoremstyle{definition}
\newtheorem{defn}[equation]{Definition}

\newtheorem{ex}[equation]{Example}

\theoremstyle{remark}
\newtheorem{rem}[equation]{Remark}

\newcommand{\supp}{\operatorname{supp}}

\newcommand{\xra}[1]{\xrightarrow{#1}}

\def\g{\gamma}

\def\z{\zeta}

\def\len{\operatorname{length}}

\def\codim{\operatorname{codim}}

\def\cPsi{\psi_{\mathrm{cyc}}}

\def\and{{ \text{ and } }}

\def\Op{\operatorname{Op}}
\def\res{\operatorname{res}}
\def\ind{\operatorname{ind}}

\def\sign{{\mathrm{sign}}}
\def\naive{{\mathrm{naive}}}
\def\l{\lambda}

\newcommand{\Kos}{\operatorname{Kos}}

\begin{document}
\begin{abstract} 
Let $Q$ be a commutative, Noetherian ring and $Z \subseteq \Spec(Q)$ a closed subset. 
Define $K_0^Z(Q)$ to be the Grothendieck group of those bounded complexes of finitely generated projective $Q$-modules that have  homology supported on $Z$.
We develop ``cyclic'' Adams operations on $K_0^Z(Q)$ 
and we prove these operations satisfy the four axioms used by Gillet and Soul\'e in \cite{GS87}. From this we recover a shorter proof of Serre's Vanishing Conjecture. We also
show our cyclic Adams operations agree with the Adams operations defined by Gillet and Soul\'e in certain cases.

Our definition of the cyclic Adams  operators is inspired by a formula due to Atiyah 
\cite{Ati66}. They have also been introduced and studied before by Haution \cite{Haution}.
\end{abstract}

\title{Cyclic Adams operations}
\author{Michael K. Brown}
\address{Hausdorff Center for Mathematics, Villa Maria, Endenicher Allee 62, D-53115 Bonn, Germany}
\email{mbrown@math.uni-bonn.de}

\author{Claudia Miller}
\address{Mathematics Department, Syracuse University, Syracuse, NY 13244-1150, USA}
\email{clamille@syr.edu}

\author{Peder Thompson}
\address{Department of Mathematics, University of Nebraska, Lincoln, NE 68588-0130, USA}
\email{pthompson4@math.unl.edu}

\author{Mark E. Walker}
\address{Department of Mathematics, University of Nebraska, Lincoln, NE 68588-0130, USA}
\email{mark.walker@unl.edu}

\thanks{This work was partially supported by a grant from the
 Simons Foundation 
(\#318705 for Mark Walker) 
and grants from the National Science Foundation 
(NSF Award DMS-0838463 for Michael Brown and Peder Thompson and 
NSF Award DMS-1003384 for Claudia Miller).}

\maketitle
\tableofcontents

\section{Introduction}

In 1987, Gillet and Soul\'e \cite{GS87}
developed a theory of Adams operations on the Grothendieck group of chain complexes of locally free coherent sheaves on a scheme that satisfy a support
condition, 
and they proved these operations 
satisfy the four key axioms (A1)--(A4) listed below.
As a major application of this theory, Gillet and Soul\'e  proved Serre's Vanishing Conjecture in full generality. (Serre \cite{Ser65} proved this in many
cases, 
and Roberts \cite{Rob85} also proved the general case, independently and at about the same time as Gillet and Soul\'{e}.) The goal of this paper is to develop 
an alternative, simpler, notion of Adams operations on such Grothendieck groups, in certain important cases, one which is based on an idea due to Atiyah
\cite{Ati66}. 
As a consequence, we arrive at a proof of the full case of Serre's Vanishing Conjecture that is considerably shorter than the proofs of Gillet-Soul\'e or Roberts.

In more detail, suppose $X$ is a separated, Noetherian scheme and $Z \subseteq X$ is a closed subset; define $K_0^Z(X)$ to be the Grothendieck group of bounded
complexes of locally free coherent sheaves 
on $X$ whose homology is supported on $Z$. This is the abelian group generated by the isomorphism classes of such complexes, modulo relations coming from short
exact sequences and quasi-isomorphisms. 
For each integer $k \geq 1$, Gillet  and Soul\'e define an operator
$$
\psi^k_{GS}: K_0^Z(X) \to K_0^Z(X)
$$
and they prove \cite[Prop 4.12]{GS87} that this operator satisfies the following axioms:
\begin{enumerate}
\item[(A1)] $\psi_{GS}^k$ is a homomorphism of abelian groups for all $X$ and $Z$;
\item[(A2)]  $\psi_{GS}^k$  is multiplicative: if $Z, W$ are closed subsets of $X$, $\a \in K_0^Z(X)$, and $\b \in K_0^W(X)$, then 
$$
\psi_{GS}^k(\a \cup \b) = \psi_{GS}^k(\a) \cup \psi_{GS}^k(\b) \in K_0^{Z \cap W}(X),
$$
where $- \cup -$ is the pairing induced by tensor product of complexes;

\item[(A3)]  
$\psi_{GS}^k$ is functorial, in the sense that, given a morphism 
$\phi: Y \to X$ and closed subsets $W \subseteq Y$ and $Z \subseteq X$ such that $\phi^{-1}(Z) \subseteq W$, 
we have an equality 
$$
\psi_{GS}^k \circ \phi^* = \phi^* \circ \psi_{GS}^k
$$
of maps $K_0^Z(X) \to K_0^W(Y)$; and 
\item[(A4)]  
if $Q$ is a commutative Noetherian ring with unit and $a \in Q$ is a non-zero-divisor, 
then
$$
\psi_{GS}^k([K(a)]) = k[K(a)] \in K_0^{V(a)}(\Spec Q),
$$
where $K(a) := (\cdots \to 0 \to  Q \xra{a} Q \to 0 \to \cdots)$ is the Koszul complex on $a$.
\end{enumerate}

Serre's Vanishing Conjecture follows from the existence of such an operator for any one value of $k \geq 2$; see
\cite[\S 5]{GS87}.

Gillet and Soul\'e's
construction of the operator $\psi_{GS}^k$ involves first establishing $\lambda$-operations, $\l^k$ for all $k \geq 1$,  on $K_0^Z(X)$. These are defined  
using the Dold-Puppe construction \cite{Dol58,DP58} (see also \cite{Kan58}) of exterior powers of chain complexes concentrated in non-negative degrees. In detail, if $\cE$ is a bounded complex of locally free coherent sheaves on $X$ that is supported on $Z$ and 
concentrated in non-negative degrees (i.e., $\cE_i = 0$ for $i < 0$), we let
$K(\cE)$ be the associated simplicial sheaf given by the Dold-Puppe functor $K$. 
Write $\Lambda^k_{\cO_X} K(\cE)$ for the simplicial sheaf obtained by applying
$\Lambda_{\cO_X}^k(-)$ degreewise to $K(\cE)$. Let $N(\Lambda_{\cO_X}^k K(\cE))$ be the chain complex given by applying the normalized chain complex
functor $N$. 
Gillet-Soul\'e \cite[\S4]{GS87} prove that, for all closed subsets $Z$ of $X$
and all integers $k \geq 0$,  there is a function
$$
\lambda^k_{GS}: K_0^Z(X) \to K_0^Z(X) 
$$
such that, if $\cE$  is concentrated in non-negative degrees, 
then 
$$
\lambda^k_{GS}([\cE]) = [N(\Lambda_{\cO_X}^k K(\cE))].
$$ 
Moreover, 
they prove that the operations $\lambda^k_{GS}$, $k \geq 1$, 
make $\bigoplus_Z K_0^Z(X)$ into a (special) lambda ring.
The operator $\psi_{GS}^k$ is then defined, as is customary, to be $Q_k(\l^1, \dots,
\l^k)$ where $Q_k$ is the $k$-th Newton polynomial.

In this paper, we build operations that satisfy the four axioms (A1)--(A4) using a simpler construction, albeit one that exists  only in a somewhat restrictive
setting.
In detail, we fix a prime $p$ and assume
$X = \Spec(Q)$, for a commutative Noetherian ring $Q$ such that $Q$ contains $\frac{1}{p}$ and all the $p$-th roots of unity. For any closed subset $Z$ of
$\Spec(Q)$, we construct the {\em $p$-th cyclic Adams operator}, which is a function 
$$
\cPsi^p: K_0^Z(Q) \to K_0^Z(Q),
$$
characterized by the following property:
if $F$ is a bounded complex of finitely generated projective $Q$-modules, then
$$
\cPsi^p([F]) =  [T^p(F)^{(1)}] - [T^p(F)^{(\zeta)}].
$$
Here, $T^p(F)$ denotes the $p$-th tensor power of the complex $F$, equipped with the canonical, signed action of the symmetric group $\Sigma_p$, 
and the superscript  ${}^{(w)}$ denotes the eigenspace of eigenvalue $w$ for the action of the $p$-cycle $(1 \, 2 \, \cdots \, p) \in \Sigma_p$. 
In particular, the definition of $\cPsi^p$ bypasses entirely the construction of $\lambda$-operations. 
The idea for the definition of $\cPsi^p$
goes back to Atiyah \cite{Ati66} (see also Benson
\cite{Ben84} and End \cite{End69}).

One of our main results is:

\begin{introthm} 
(See Theorem \ref{psi-GS-axioms}  for the precise statement.)
For any prime $p$, the $p$-th cyclic Adams operation $\cPsi^p$ satisfies 
 the four Gillet-Soul\'e axioms (A1)--(A4) on the category of affine schemes $\Spec(Q)$ with the property that $Q$ contains $\frac{1}{p}$ and all $p$-th roots of unity.
\end{introthm}

The hypotheses involving the prime $p$  are not significant restrictions for many purposes. Note that if $Q$ is local, then $p$ is invertible in $Q$
for all primes other than the residue characteristic. Moreover, the requirement that $Q$ contain the $p$-th roots of unity is a mild one, since adjoining such
roots gives an \'etale  extension of $Q$. In particular, 
Serre's Vanishing Conjecture is a direct consequence of the above Theorem, via the same argument used by Gillet and Soul\'e;
see Corollary \ref{cor728c}.

\edit{Added paragraph}
Since the first version of this paper was made publicly available, 
we have learned that the operator $\cPsi^p$ has also been defined previously by Haution in his thesis \cite{Haution}. Haution
works more generally over schemes, but just for schemes defined over a ground field $k$ not of characteristic $p$. (We do believe, however, that most of his
proofs go through under the more general context of schemes over $\Z[\frac{1}{p}, e^{\frac{2 \pi i}{p}}]$.) In his thesis, Haution establishes the existence,
naturality, additivity and multiplicativity of these operators --- i.e., he establishes axioms (A1)--(A3) in the above list. He does not establish (A4), and his
approach is different in that he bypasses introducing the power operations that we develop in Section \ref{sectwo}. These power operations seem to be
necessary for the proof of (A4), and also for our proof of the commutativity of the cyclic Adams operations in Section \ref{seccommute}. 
For these reasons, we have kept this paper mostly unchanged from the original version, but we have added careful indications of which results presented here can
also be found in \cite{Haution}.

In addition to the results described  above, we also address the issue of whether the operator $\cPsi^p$ agrees with the $p$-th Adams operation of Gillet-Soul\'e.
We believe that they coincide whenever both are defined, but are only able to prove it in the case 
that $p!$ is invertible in $Q$; see Corollary \ref{cor930}. 

In developing the proof of Corollary \ref{cor930}, we also show that if $k!$ is invertible in $Q$,
then Gillet and Soul\'e's  operation $\l^k$ may be defined by taking ``naive'' exterior powers of complexes; see Theorem \ref{gs=hashimoto}. 
This fact is a ``folklore'' result (see the discussion at the beginning of \S \ref{E^k}), but we provide a careful proof here.

In addition to their simplicity, 
another advantage the operators $\cPsi^p$ have over the operators defined by Gillet-Soul\'e is that their definition ports well to other contexts where the
Dold-Puppe functors are unavailable.  In a forthcoming
paper we establish the existence of analogously defined  operators $\cPsi^p$ on the $K$-theory of matrix factorizations, and we 
prove that the analogues of the four Gillet-Soul\'e axioms hold. Using
these properties, we prove a conjecture of Dao and Kurano \cite[3.1 (2)]{DK12} concerning the vanishing of the $\theta$-invariant.

We thank Luchezar Avramov for helpful conversations in preparing this document, Dave Benson for leading us to his relevant paper \cite{Ben84} as well as
answering some of our questions, and Paul Roberts for sharing his unpublished notes \cite{RobertsUnpub}
describing Hashimoto's result mentioned in the beginning of \S \ref{E^k}. We also thank Olivier Haution for drawing our attention to his thesis.

\section{Tensor power operations} \label{sectwo}

Let $Q$ be a Noetherian, commutative ring, 
$Z \subseteq \Spec(Q)$ a closed subset, and $G$ a finite group. 
Let $\cP^Z(Q; G)$ denote the category of bounded complexes of finitely generated projective $Q$-modules
with homology supported on $Z$ and equipped with a left $G$-action (with $G$ acting via chain maps). Morphisms are $G$-equivariant chain maps.
Equivalently, $\cP^Z(Q; G)$ consists of bounded complexes of left $Q[G]$-modules which, upon restricting scalars along $Q \subseteq Q[G]$, are complexes of finitely
generated projective $Q$-modules supported on $Z$. 

Let $K_0^Z(Q;G)$ denote the Grothendieck group of $\cP^Z(Q;G)$, defined to be the group generated by isomorphism classes of objects modulo the relations 
$$
[X] = [X'] + [X'']
$$
if there exists an (equivariant) short exact sequence $0 \to X' \to X \to X'' \to 0$ and
$$
[X] = [Y]
$$
if there exists an (equivariant) quasi-isomorphism joining $X$ and $Y$.  
Observe that the group operation is realized by direct sum of complexes:
$[X] + [Y] = [X \oplus Y]$.

We write $\cP^Z(Q)$ and $K_0^Z(Q)$ when $G$ is the trivial group.

\begin{rem}\label{remK0andP} 
$K_0^Z(Q;G)$ can equivalently be described as the abelian monoid of isomorphism classes of objects of $\cP^Z(Q; G)$, under the operation of direct
  sum, modulo the 
  two relations above. For observe that for any $X \in \cP^Z(Q;G)$, we have the short exact sequence $0 \to X \to \cone(X \xra{=} X) \to \Sigma(X) \to 0$, where
$\Sigma(X)$ denotes the suspension of $X$. It follows that $[X] + [\Sigma(X)] = 0$, and hence that this monoid is an abelian group. 

In particular, $K_0^Z(Q;G)$ has the following universal mapping property: given an abelian monoid $M$ and an assignment of an element $(X) \in M$ to each
object $X$ of $\cP^Z(Q; G)$ such that $(0) = 0$, 
$(X) = (X') + (X'')$ 
if there exists a short exact sequence $0 \to X' \to X \to X'' \to 0$, and $(X) = (Y)$ if $X$ and $Y$ are quasi-isomorphic, then there exists a unique group
homomorphisms $K_0^Z(Q; G) \to U(M)$ sending $[X]$ to $(X)$, where $U(M)$ denotes the group of units of $M$. 
\end{rem}

Tensor product over $Q$, with the group action given by the diagonal action, induces a pairing on $K_0^Z(Q;G)$ making it into a non-unital ring. 
If $Z = \Spec(Q)$, then $K_0^{\Spec Q}(Q; G)$ is a unital ring, with $1 = [Q]$, 
and there is a ring isomorphism
$$
R_Q(G) \xra{\cong} K_0^{\Spec Q}(Q; G),
$$
where $R_Q(G)$ denotes the representation ring of $G$ with $Q$ coefficients: 
By definition, $R_Q(G)$ is the abelian group generated by isomorphism classes of
projective $Q$-modules equipped with a $G$-action, modulo relations coming from short exact sequences. The isomorphism sends the class of a representation $\rho: G \to \Aut_Q(P)$ to
the class of the evident complex concentrated in degree $0$. The inverse map sends the class of a complex to the alternating sum of the classes of its
components.  
As a special case of this, we have
$K_0(Q) \cong K_0^{\Spec Q}(Q)$.

For any $n \geq 1$, let $\Sigma_n$ denote the group of permutations of the set $\{1, \dots, n\}$. 
For $n \geq 1$ and $X \in \cP^Z(Q;G)$, let $T^n(X) \in \cP^Z(Q; G \times \Sigma_n)$ be the complex $\overbrace{X \otimes_Q \cdots \otimes _Q X}^n$ equipped with
the diagonal $G$-action,
$$
g(x_1 \otimes \cdots \otimes x_n) = g(x_1) \otimes \cdots \otimes g(x_n),
$$
and equipped with a 
$\Sigma_n$-action  given by
$$
\sigma(x_1 \otimes \cdots \otimes x_n) = 
\pm \, x_{\sigma(1)} \otimes \cdots \otimes x_{\sigma(n)},
$$
where the sign is uniquely determined by the following rule: if $\sigma$ is the adjacent transposition $(i \; i+1)$ for some $1 \leq i \leq n-1$, then 
$$
\sigma(x_1 \otimes \cdots \otimes x_n) = 
(-1)^{|x_i||x_{i+1}|} x_1 \otimes \cdots \otimes x_{i-1}  \otimes x_{i+1} 
\otimes x_i \otimes x_{i+2} \otimes \cdots \otimes x_n.
$$

For $0 \leq i \leq n$, let $\Sigma_{i,n-i}$ denote the subgroup of $\Sigma_n$ consisting of permutations that stabilize the subsets $\{1, 2, \dots, i\}$ and
$\{i+1, i+2,  \dots, n\}$. 
We identify $\Sigma_{i,n-i}$ with  $\Sigma_i \times \Sigma_{n-i}$ in the obvious way.  
If  $X \in \cP^Z(Q; \Sigma_{i,n-i} \times G)$ 
then $Q[\Sigma_n] \otimes_{Q[\Sigma_{i,n-i}]} X$ is in $\cP^Z(Q;\Sigma_n \times G)$.

\begin{thm} 
\label{thm-t^k}
For any $Q$, $Z$, $G$, and $n \geq 1$ as above, there is a function
$$
t^n_{\Sigma}: K_0^Z(Q;G) \to K_0^Z(Q; G \times \Sigma_n)
$$
such that 
$$
t^n_\Sigma([X]) = [T^n(X)]
$$
for any object $X$ of $\cP^Z(Q;G)$. Moreover, 
if $0 \to X' \to X \to X'' \to 0$ is a short exact sequence of objects of $\cP^Z(Q; G)$, then
$$
t^n_\Sigma([X]) = \sum_{i=0}^n \left[  Q[\Sigma_n] \otimes_{Q[\Sigma_{i, n-i}]} T^i(X') \otimes_Q T^{n-i}(X'') \right].
$$
\end{thm}

\begin{rem} See \cite[II.3.4 and II.3.8]{Haution} for a similar result involving
  direct sums of complexes and for a method of reducing the case of a
  short exact sequence to a direct sum. 
\end{rem} \edit{added remark}

The proof of the Theorem occupies the remainder of this section.

\begin{lem} \label{lemstar}
The bi-functor 
$$
\cP^Z(Q; \Sigma_i \times G) \times 
\cP^Z(Q; \Sigma_j \times G)
\to 
\cP^Z(Q; \Sigma_{i+j} \times G)
$$
sending $(X,Y)$ to $Q[\Sigma_{i+j}] \otimes_{Q[\Sigma_{i,j}]} X \otimes_Q Y$, equipped with the diagonal $G$-action, induces a bilinear pairing
$$
\star = \star_{i,j}: K_0^Z(Q; \Sigma_i\times G) \times
K_0^Z(Q; \Sigma_j \times G) \to
K_0^Z(Q; \Sigma_{i+j}\times G).
$$
The pairing is associative and commutative, in the sense that
$$
(a \star_{i,j} b) \star_{i+j,k} c = 
a \star_{i, j+k} (b \star_{j,k} c)
$$
and
$$
a \star_{i,j} b = b \star_{j,i} a.
$$
\end{lem}
\begin{rem}
This ``star pairing'' is related to pairings considered by 
Atiyah \cite[\S 1]{Ati66} and Knutson \cite[p. 127]{Knu73}. 
See the discussion in \S \ref{agreement}.
\end{rem}

\begin{proof}[Proof of Lemma \ref{lemstar}]  
Note that $Q[\Sigma_{i+j}]$ is a flat $Q[\Sigma_{i,j}]$-module, and hence this functor preserves short exact sequences and quasi-isomorphisms in each
argument. It thus induces a bilinear pairing on Grothendieck groups as indicated. 

Associativity holds since there is an isomorphism in
$\mathcal{P}^Z(Q; \Sigma_{i+j+k} \times G)$ from
$$Q[\Sigma_{i+j+k}] \otimes_{Q[\Sigma_{i+j,k}]}(Q[\Sigma_{i+j}]
\otimes_{Q[\Sigma_{i,j}]} X \otimes_Q Y) \otimes_Q Z$$ 
to 
$$Q[\Sigma_{i+j+k}] \otimes_{Q[\Sigma_{i,j+k}]} X \otimes_Q (Q[\Sigma_{j+k}]
\otimes_{Q[\Sigma_{j,k}]} Y \otimes_Q Z)$$
given by
$$\sigma \otimes \omega \otimes x \otimes y \otimes z \mapsto
\sigma \omega \otimes x \otimes 1 \otimes y \otimes z.$$

As for commutativity, let $\tau:=(1 \,
2 \, \cdots \, i+j)^j \in \Sigma_{i+j}$, and let $h$ denote the
automorphism of $\Sigma_{i+j}$ given by $\sigma \mapsto \tau \sigma \tau^{-1}$. Notice
that $h$ restricts to an isomorphism
$$\Sigma_{i,j} \xrightarrow{\cong} \Sigma_{j,i},$$
and, moreover, this isomorphism coincides with the map
given by the composition of evident isomorphisms
$$\Sigma_{i,j} \xrightarrow{\cong} \Sigma_i \times \Sigma_j \xrightarrow{\cong} \Sigma_j \times
\Sigma_i \xrightarrow{\cong} \Sigma_{j,i}.$$

It follows that one has an
isomorphism in $\mathcal{P}^Z(Q; \Sigma_{i+j} \times G)$
$$Q[\Sigma_{i+j}] \otimes_{Q[\Sigma_{i,j}]} X \otimes_Q Y \xrightarrow{\cong}
Q[\Sigma_{i+j}] \otimes_{Q[\Sigma_{j,i}]} Y \otimes_Q X$$
that sends elements of the form $\sigma \otimes x \otimes y$, where
$\sigma \in \Sigma_{i+j}$, to $\sigma \tau^{-1} \otimes y \otimes x$.
\end{proof}

\begin{lem} \label{lem712a}
Given a short exact sequence $0 \to X' \to X \to X'' \to 0$ in $\cP^Z(Q; G)$, for any $n \geq 1$ there is a filtration 
$$
0 = F_{-1} \subseteq F_0 \subseteq F_1  \subseteq \cdots \subseteq F_n = T^n(X)
$$
in $\cP^Z(Q; \Sigma_n \times G)$ such that 
$$
F_i/F_{i-1} \cong Q[\Sigma_n] \otimes_{Q[\Sigma_{n-i, i}]} T^{n-i}(X') \otimes_Q T^i(X'').
$$
Consequently, in $K_0^Z(Q; \Sigma_p \times G)$ we have
$$
[T^n(X)] = \sum_i [T^{n-i}(X')] \star_{n-i,i} [T^i(X'')].
$$
\end{lem}
  
\begin{proof}
We identify $X'$ as a subcomplex of $X$. Define $F_i$ as the image of
$$
\alpha_i: Q[\Sigma_n] \otimes_Q T^{n-i}(X') \otimes_Q T^i(X) 
\to
T^n(X)
$$
sending $\sigma \otimes x_1 \otimes \cdots \otimes x_n$ to $\sigma(x_1 \otimes \cdots \otimes x_n)$, where $x_1, \dots, x_{n-i} \in X'$.
In other words, $F_i$ is the closure under the action of $\Sigma_n$ of the image of the canonical map 
$T^{n-i}(X') \otimes_Q T^i(X) \to T^n(X)$.

The map $\alpha_i$ factors as
$$
Q[\Sigma_n] \otimes_Q 
T^{n-i}(X') \otimes_Q T^i(X) 
\onto
Q[\Sigma_n] \otimes_{Q[\Sigma_{n-i,i}]} 
T^{n-i}(X') \otimes_Q T^i(X) 
\xra{\overline{\a}_i}
T^n(X).
$$
Also, the restriction of $\alpha_i$ to the subcomplex
$Q[\Sigma_n] \otimes_Q T^{n-i+1}(X') \otimes_Q T^{i-1}(X)$ coincides with $\alpha_{i-1}$.
We have a right exact sequence
$$
\begin{aligned}
Q[\Sigma_n] \otimes_Q & T^{n-i+1}(X') \otimes_Q T^{i-1}(X) 
\to
Q[\Sigma_n]  \otimes_{Q[\Sigma_{n-i,i}]} 
T^{n-i}(X') \otimes_Q T^i(X) \\
& \to
Q[\Sigma_n] 
\otimes_{Q[\Sigma_{n-i,i}]} 
T^{n-i}(X') \otimes_Q T^i(X'') \to 0.
\end{aligned}
$$
These facts imply the existence of a surjective map
\begin{equation} \label{E720}
Q[\Sigma_n] \otimes_{Q[\Sigma_{n-i,i}]}  T^{n-i}(X') \otimes_Q T^i(X'') \onto
F_i/F_{i-1}
\end{equation}
and it remains to prove it is injective too. 

We may assume $\Spec(Q)$ is connected, so that each complex 
$X', X, X''$ has well-defined total rank $r', r, r''$, respectively, where we define total rank to be 
to be the sum of the ranks of all the components of a complex.  Moreover, we have $r = r' +
r''$. Then the total  rank of
$$
Q[\Sigma_n] \otimes_{Q[\Sigma_{n-i,i}]} 
T^{n-i}(X') \otimes_Q T^i(X'') 
$$
is $((r')^{n-i} (r'')^i) {n \choose i}$. Observe that
$$
\sum_i ((r')^{n-i} (r'')^i) {n \choose i}
=
(r' + r'')^n = r^n.
$$
But the sum of the ranks of the complexes $F_i/F_{i-1}$ is also $r^n$, since they are the associated graded modules associated to a filtration of $T^n(X)$. It follows that each
map (\ref{E720}) 
must be injective too.   
\end{proof}

We define a multiplicative abelian monoid $M$ as follows. As a set, $M$ is 
$$
\{1\} \times \prod_{i=1}^\infty K_0^Z(Q; \Sigma_i \times G) z^i,
$$
the collection of power series in $z$ of the form $1 + \a_1 z + \a_2 z^2 + \cdots $ with $\a_i \in K_0^Z(Q; \Sigma_i \times G)$ for all $i$.
We define a multiplication rule on $M$ using the $\star$ pairings:
$$
\left(\sum_i \alpha_i z^i\right) \star \left(\sum_j \beta_j z^j\right)
:= \sum_{l} \sum_{i+j = l} (\alpha_i \star_{i,j} \beta_j) z^l,
$$
where by convention $\a_0 = 1$, $\b_0 = 1$, $\alpha_0 \star \beta_j = \beta_j$, 
and $\alpha_i \star \b_0 = \a_i$. 
The associative and commutative properties of $\star$ given in Lemma \ref{lemstar} imply that $(M, \star)$ is an abelian monoid.

For $X \in \cP^Z(Q;G)$, define $t(X) \in M$ by 
$$
t(X) = 1 + [X]z + [T^2(X)]z^2 + \cdots
$$
By Lemma \ref{lem712a}, $t(X) = t(X') \star t(X'')$ if $0 \to X' \to X \to X'' \to 0$ is a short exact sequence in $\cP^Z(Q; G)$. 
If $X \xra{\sim} X'$ is a quasi-isomophism in $\cP^Z(Q;G)$,  
then the induced map $T^i(X) \to T^i(X')$ is also a quasi-isomorphism for all $i$, 
and hence $t(X) = t(X')$. 
By Remark \ref{remK0andP}, we get an induced group homomorphism 
$$
t: K_0^Z(Q;G) \to U(M)
$$
landing in the group of units of $M$.

The function $t^n_\Sigma$ is defined to be the composition of $t$ with the function $U(M) \to K_0^Z(Q; \Sigma_n \times G)$ sending a power series to its $z^n$
coefficient. The first assertion of Theorem \ref{thm-t^k} 
follows, and the second is a consequence of Lemma \ref{lem712a}.

\section{Cyclic Adams operations}

We define  a ``cyclic'' Adams operation, $\cPsi^p$, on $K_0^Z(Q)$ for each prime $p$. 
The definition is motivated by an observation of Atiyah \cite[2.7]{Ati66}; 
see also Benson \cite{Ben84} and End \cite{End69}. In the case $p = 2$, the operator $\cPsi^2$ was defined and developed in unpublished work of P.\ Roberts
\cite{RobertsUnpub}, who in turn credited the idea to unpublished work of M.\ Hashimoto and M.\ Nori. Finally, as mentioned in the introduction, these operators
have also been defined and developed by Haution \cite{Haution} when $Q$ contains a field of characteristic different than $p$. \edit{Added sentence}

Throughout this section, assume $p$ is a prime and $Q$ is an $A_p$-algebra, where $A_p$ is the subring of $\C$ defined by
$$
A_p = \Z\left[\scriptstyle{\frac{1}{p}}, e^{\scriptscriptstyle{\frac{2 \pi i}{p}}}\right],
$$

Define $C_p$ to be the subgroup of $\Sigma_p$ generated 
by the $p$-cycle $\s := (1 \, 2 \, \cdots \, p)$. 
For a $p$-th root of unity $\zeta$ (including the case $\zeta = 1$), let $Q_{\zeta}$ 
denote the $Q[C_p]$-module $Q$ equipped with the $C_p$-action $\s q = \zeta q$. 
Since $Q$ is an $A_p$-algebra, we have $Q[C_p] \cong \bigoplus_{\zeta} Q_{\zeta}$ as $Q[C_p]$-modules.

For $Y \in \cP^Z(Q; C_p)$, define
$$
Y^{(\zeta)} := \Hom_{Q[C_p]}(Q_\zeta, Y) = 
\ker(\s - \zeta: Y \to Y).
$$
Since $Q_\zeta$ is a projective $Q[C_p]$-module, $Y \mapsto Y^{(\zeta)}$ is an exact functor, and hence it induces a map
$$
\phi^p_\zeta: K_0^Z(Q; C_p) \to K_0^Z(Q). 
$$

\begin{prop} 
Assume $Q$ is an $A_p$-algebra.  For each $p$-th root of unity $\zeta$, there is a function
$$
t^p_\zeta: K_0^Z(Q) \to K_0^Z(Q)
$$
with $t^p_\zeta([X]) = [T^p(X)^{(\zeta)}]$.
\end{prop}

\begin{proof} 
Restriction along the inclusion $C_p \into \Sigma_p$ determines a map
$$
K_0^Z(Q; \Sigma_p) \xra{\res} K_0^Z(Q; C_p).
$$
We define $t^p_\zeta$ to be the composition of
$$
K_0^Z(Q) \xra{t^p_\Sigma} K_0^Z(Q; \Sigma_p) \xra{\res} K_0^Z(Q; C_p)  \xra{\phi^p_\zeta} K_0^Z(Q).
$$
\end{proof}

\begin{defn}
For an $A_p$-algebra $Q$ and closed subset $Z$ of $\Spec(Q)$, define the function
$$
\cPsi^p: K_0^Z(Q) \to K_0^Z(Q) \otimes_\Z \Z[e^{\frac{2 \pi i}{p}}]
$$
by 
$$
\cPsi^p := \sum_{\zeta}  \zeta t^p_\zeta
$$
where the sum is taken over all $p$-th roots of unity $\zeta$. Thus for $X \in \cP^Z(Q)$, 
$$
\cPsi^p([X]) = \sum_{\zeta} \zeta [T^p(X)^{(\zeta)}].
$$
\end{defn}

In view of the following lemma, the map $\cPsi^p$  is independent of the generator chosen for $C_p$.  
The lemma is proven in \cite[II.3.6]{Haution}, but we include the details here. \edit{Added citation.}

\begin{lem} \label{lem:indofzeta}
Assume $Q$ is an $A_p$-algebra. 
If $\zeta$ and $\zeta'$ are both primitive $p$-th roots of unity, then
$$
[T^p(X)^{(\zeta)}] = [T^p(X)^{(\zeta')}] \in K_0^Z(Q)
$$
for all $X \in \cP^Z(Q)$. 
\end{lem}

\begin{proof} 
We show $\res(Y)^{(\zeta)}$ and  $\res(Y)^{(\zeta')}$ are isomorphic 
objects of $\cP^Z(Q)$ for any $Y \in \cP^Z(Q; \Sigma_p)$, where $\res:
  \cP^Z(Q; \Sigma_p) \to \cP^Z(Q; C_p)$ is the restriction functor.  

Note that $\zeta' = \zeta^i$ for some $1 \leq i \leq
  p-1$. Let $\tau \in \Sigma_p$ be a permutation such that $\tau^{-1} \sigma \tau = \sigma^i$. (Recall $\sigma = (1 \, 2 \, \cdots \, p)$.) Then $\tau$
  determines an isomorphism from $\res(Y)$ to $\res'(Y)$, where $\res'$ is restriction along $C_p \xra{\s \mapsto \s^i} C_p \subseteq \Sigma_p$. We have
  $\res'(Y)^{(\zeta)} = \res(Y)^{(\zeta^i)}$.
\end{proof}

\begin{rem}
More generally, for any integer $n \geq 1$,
$[T^n(X)^{(\zeta)}] = [T^n(X)^{(\zeta')}]$ holds in $K_0^Z(Q)$, 
as long as $\zeta$ and $\zeta'$ are $n$-th roots of unity of the same order. 
\end{rem}

Since $\displaystyle{\sum_{\zeta \ne 1} \zeta = -1}$, we deduce from the Lemma:

\begin{cor}\label{psi-simplified} 
Assume $Q$ is an $A_p$-algebra and let $\zeta$ be a primitive $p$-th root of unity.
We have 
$$
\cPsi^p([X]) = [T^p(X)^{(1)}] - [T^p(X)^{(\zeta)}].
$$ 
\end{cor}

The corollary shows, in particular, that $\cPsi^p$ takes values in 
$K_0^Z(Q)$ viewed as a subgroup of $K_0^Z(Q) \otimes_\Z \Z[e^{\frac{2 \pi i}{p}}]$, and we will henceforth view $\cPsi^p$ as a function of the form
\begin{equation} \label{E930}
\cPsi^p: K_0^Z(Q) \to K_0^Z(Q).
\end{equation}

\begin{thm} 
\label{psi-GS-axioms}
Fix a prime $p$. 
The operation $\cPsi^p$ satisfies the Gillet-Soul\'e axioms of being an ``Adams operation of degree $p$'' on the category of commutative, Noetherian $A_p$-algebras.
That is, letting $Q$ and $R$ be commutative, Noetherian $A_p$-algebras, we have:
\begin{enumerate}
\item $\cPsi^p$ is a group endomorphism of $K_0^Z(Q)$ for all
closed subsets $Z$ of $\Spec(Q)$. 

\item Given
$\a \in K_0^Z(Q)$ and $\b \in K_0^W(Q)$ for closed subsets $Z$ and  $W$ of $\Spec(Q)$,  we have
$$
\cPsi^p(\a \cup \b) = \cPsi^p(\a) \cup \cPsi^p(\b) \in K_0^{Z \cap W}(Q),
$$
where $- \cup -$ is the pairing determined by tensor product over $Q$.

\item 
Given a morphism of affine schemes 
$$
\phi: \Spec(R) \to \Spec(Q)
$$ 
over $\Spec(A_p)$ and given closed subsets 
$W \subseteq \Spec(R)$ and $Z \subseteq \Spec(Q)$ such that $\phi^{-1}(Z) \subseteq W$, 
we have an equality 
$$
\cPsi^p \circ \phi^* = \phi^* \circ \cPsi^p
$$
of maps $K_0^Z(Q) \to K_0^W(R)$. 

\item 
If $a = (a_1, \dots, a_n)$ is any sequence of elements in $Q$ 
and $K(a)$ is the associated Koszul complex,
viewed as an object of $\cP^{V(a_1, \dots, a_n)}(Q)$, we have
$$
\cPsi^p([K(a)]) = p^n[K(a)] \in K_0^{V(a_1, \dots, a_n)}(Q).
$$
\end{enumerate}
\end{thm}

\begin{rem} The Gillet-Soul\'e axioms include non-affine schemes too,
but we won't require that level of generality. Also, their fourth axiom assumes $a$ is a regular sequence, but the property holds more generally for any such
sequence, both for our operators  and 
theirs.
\end{rem} 

\begin{rem}
Proofs of (1)--(3) of the theorem can be found in Haution's work under the additional assumption that $Q$ contains a field of characteristic different than $p$.  
Specifically \cite[II.3.8]{Haution} proves $\cPsi^p$ is additive and
\cite[II.3.10]{Haution} proves (2) and (3). We believe his proofs  apply verbatim  to the slightly
more general setting of this paper. Nevertheless, for the sake of making this paper self-contained, we will include proofs of (1)--(3). \edit{added remark}
\end{rem}

\begin{proof}
By construction, $\cPsi^p$ factors as
$$
K_0^Z(Q) \xra{t^p_\Sigma} K_0^Z(Q; \Sigma_p) \xra{\res} 
K_0^Z(Q; C_p) \xra{\phi^p} K_0^Z(Q)
$$
where  $\phi^p = \sum_\zeta \zeta \phi^p_\zeta$.
For $Y \in \cP^Z(Q; C_p)$, let us say $Y$ is {\em extended} if $Y \cong Y' \otimes_Q Q[C_p]$ for some $Y' \in \cP^Z(Q)$.

The following result may also be found in \cite[II.3.7]{Haution}.
\edit{added reference}

\begin{lem} \label{lem2}
If $Y \in \cP^Z(Q; C_p)$ is extended, then $\phi^p([Y]) = 0$.
\end{lem}

\begin{proof} 
If $Y$ is extended, 
$$
Y^{(\zeta)} \cong
\Hom_{Q[C_p]}(Q_\zeta, Y' \otimes_Q Q[C_p]) \cong 
\Hom_{Q[C_p]}(Q_\zeta, Q[C_p]) \otimes_Q Y' \cong 
Y'
$$ 
as objects of $\cP^Z(Q)$, since $Q[C_p] = \bigoplus_\zeta Q_\zeta$ and 
$\Hom_{Q[C_p]}(Q_\zeta, Q_{\zeta'})$ is $0$ for $\zeta \ne \zeta'$ and $Q$ for $\zeta = \zeta'$.
Thus
$$
\phi^p([Y]) = \left(\sum_\zeta \zeta\right) [Y'] = 0. \qedhere
$$
\end{proof}

We claim that for each $1 \leq i \leq p-1$ and $X,Y\in \cP^Z(Q)$, 
$$
Q[\Sigma_p] \otimes_{Q[\Sigma_{i,p-i}]}  T^i(X) \otimes_Q T^{p-i}(Y) 
$$
is an extended complex of $Q[C_p]$-modules. 
Granting this claim, by Theorem \ref{thm-t^k} 
$$
t^p_\Sigma([X] + [Y]) = \sum_i
\left[Q[\Sigma_p]  \otimes_{Q[\Sigma_{i,p-i}]}  T^i(X) \otimes_Q T^{p-i}(Y) \right], 
$$
and thus Lemma \ref{lem2} shows that
$$
\cPsi^p([X] + [Y]) = \cPsi^p([X]) + \cPsi^p([Y]),
$$
and part (1) of the Theorem follows.

To prove the claim, we show more generally that 
for any  $1 \leq  i \leq p-1$
and any left $Q[\Sigma_{i, p-i}]$-module $M$, the $Q[C_p]$-module
$Q[\Sigma_p] \otimes_{Q[\Sigma_{i, p-i}]} M$ is extended. 
Let $C_p\tau_1 \Sigma_{i, p-i}, \dots,  C_p\tau_m \Sigma_{i, p-i}$ 
be a set of double coset representatives in $\Sigma_p$. Since $p$ is prime, $\tau \Sigma_{i, p-i} \tau^{-1}$ intersects $C_p$ trivially for all $\tau \in
\Sigma_p$. It follows that, for each $j$, 
$$
Q[C_p \tau_j \Sigma_{i, p-i}] \cong Q[C_p \tau_j] \otimes_Q Q[\Sigma_{i, p-i}] 
$$
as 
$Q[C_p]$-$Q[\Sigma_{i, p-i}]$-bimodules. Also, one has an isomorphism
$$
Q[\Sigma_p] \cong \bigoplus_j Q[C_p \tau_j \Sigma_{i, p-i}].
$$
of $Q[C_p]$-$Q[\Sigma_{i, p-i}]$-bimodules. Combining these gives an isomorphism
$$
Q[\Sigma_p] \otimes_{Q[\Sigma_{i, p-i}]} M \cong \bigoplus_j(Q[C_p \tau_j] \otimes_Q Q[\Sigma_{i, p-i}]) \otimes_{Q[\Sigma_{i, p-i}]} M
$$
of left $Q[C_p]$-modules. But the latter is isomorphic as a left $Q[C_p]$-module to 
$$
\bigoplus_j Q[C_p]\tau_j \otimes_Q M \cong
\bigoplus_j Q[C_p]\otimes_Q M,
$$
which is extended. 

The following Lemma will be useful in proving parts (2) and (4):

\begin{lem}
\label{philemma}
If $X, Y \in \cP^Z(Q; C_p)$ then 
$\phi^p([X] \cup [Y]) = \phi^p([X]) \cup \phi^p([Y])$.
\end{lem}

\begin{proof}
We have
$$
\begin{aligned} 
\phi^p([X] \cup [Y]) & = \sum_\zeta \zeta \sum_{\zeta',\zeta^{''} \text{; }
  \zeta'\zeta'' = \zeta} [X^{(\zeta')}] \cup [Y^{(\zeta'')}] \\
& = (\sum_{\zeta'} \zeta' [X^{(\zeta')}]) \cup (\sum_{\zeta''}
\zeta^{''} [Y^{(\zeta'')}]) \\
& =\phi^p([X]) \cup \phi^p([Y]),
\end{aligned}
$$
where $\zeta$ ranges over all $p$-th roots of unity and $\zeta', \zeta^{''}$ range over all pairs 
of $p$-th roots of unity whose product is $\zeta$.
\end{proof}

We now prove (2). Suppose $X \in \cP^Z(Q)$ and  $Y \in \cP^W(Q)$, and recall $[X] \cup [Y] = [X \otimes_{Q} Y]$. 
The canonical isomorphism 
$$
T^p_{Q}(X \otimes_{Q} Y) \cong
T^p_Q(X) \otimes_{Q} T^p_Q(Y)
$$
of complexes over $Q$ preserves the $\Sigma_p$-action (with the action
on the right being the diagonal one). It thus follows from Lemma~\ref{philemma} that $\cPsi^p([X] \cup [Y]) = \cPsi^p([X]) \cup \cPsi^p([Y])$.

Assertion (3) is clear from the construction of $\cPsi^p$.

Using (2) 
it suffices to prove (4) when $n = 1$. 
Let $a \in Q$ be any element,  
and let $K= \Kos_Q(a)$ be the associated two-term Koszul complex. 
Recall $T^p(K)$ may be identified with the free commutative dg-$Q$-algebra generated by degree $1$ elements 
$e_1, \dots, e_p$ with differential $d(e_i) = a$. 
The action of $\s \in C_p$ is given by $\s(e_i) = e_{i+1}$, for $1 \leq i \leq p-1$ and $\s(e_p) = e_1$.

Note that the degree one part of $T^p(K)$ (i.e., the $Q$-span of $e_1, \dots, e_p$)  is the regular representation of $C_p$; 
we prove (4) by using a basis of eigenvectors instead of $e_1, \dots, e_p$.  
Explicitly, for each $p$-th root of unity $\zeta$ (including $\zeta = 1$), set
$$
v_\zeta = \frac{1}{p} \sum_i \zeta^{-i} e_i.
$$
Taking a full list of the $p$-th roots of unity $\zeta_0=1, \zeta_1, \dots, \zeta_{p-1}$ and setting 
$v_i = v_{\zeta_i}$, we see that the $v_0, \dots, v_{p-1}$ form a basis 
of the degree one part of $T^p(K)$, and hence we may view it as the exterior algebra on this list of elements. 
For this new basis, we have
$$
\sigma (v_i) =  \zeta_i v_i
$$
and
$$
d(v_0) = a 
{\textrm{ and }} 
d(v_i) = 0, {\textrm{ if }} i \neq 0. 
$$

Next note that $K_0^{V(a)}(Q; C_p)$ has a multiplication rule, given by tensoring over $Q$ and then using the diagonal action of $C_p$.
With this structure we have
$$
T^p(\Kos(a)) = (Qv_0 \stackrel{a}{\lra} Q) \otimes_Q (Qv_1 \stackrel{0}{\lra} Q) \otimes_Q \cdots  \otimes_Q (Qv_{p-1} \stackrel{0}{\lra} Q).
$$
Note that each of the factors on the right determine classes in $K_0^{V(a)}(Q; C_p)$, and so the tensor product here can be interpreted as
occurring in $K_0^{V(a)}(Q; C_p)$.

Recall that $\phi^p: K_0^{V(a)}(Q; C_p) \lra K_0^{V(a)}(Q)[\zeta]$ is  the function 
$$
\phi^p([Y]) = \sum_{i} \zeta_i[Y^{(\zeta_i)}]
$$  
so that $\cPsi^p = \phi^p \circ T^p$.
By Lemma~\ref{philemma}, $\phi^p([X \otimes Y]) = \phi^p([X]) \cup
\phi^p([Y])$ in $K_0(Q)[\zeta]$. It follows that 
\begin{align*}
\cPsi^p(\Kos(a)) 
&= \phi^p(T^p(\Kos(a))) \\
&= \phi^p(Qv_0 \stackrel{a}{\lra}  Q) \cup \phi^p(Qv_1 \stackrel{0}{\lra}  Q) \cup \dots \cup \phi^p(Q_{v_{p-1}} \stackrel{0}{\lra}  Q)
\end{align*}
It is clear that $\phi^p(Qv_0 \stackrel{a}{\lra}  Q) = (Q \stackrel{a}{\lra}  Q) = \Kos(a)$ and
$$
\phi^p(Qv_i \stackrel{0}{\lra} Q) = \phi^p([Q] - [Qv_i]) = [Q] - \zeta_i [Q].
$$
Hence
$$
\cPsi^p(\Kos(a)) = \Kos(a) \prod_{i=1}^{p-1} (1 - \zeta_i).
$$
Finally, observe that  $\prod_{i=1}^{p-1} (1 - \zeta_i)$ is the result of evaluating $\prod_{i=1}^{p-1} (t - \zeta_i)$ at $t = 1$ and that 
$$
\prod_{i=1}^{p-1} (t - \zeta_i) = (t^p - 1)/(t-1) = 1 + t + \cdots + t^{p-1}. 
$$
Thus $\prod_{i=1}^{p-1} (1 - \zeta_i) = p$ and we get
$$
\cPsi^p(\Kos(a)) = p \Kos(a).
$$
This completes the proof of Theorem \ref{psi-GS-axioms}.
\end{proof}

The next Corollary uses the following notation.
If $Q$ is regular, $Z \subseteq \Spec(Q)$ is a closed subset, and $M$ is a finitely generated $Q$-module with $\supp(M) \subseteq Z$, 
we write $[M] \in K_0^Z(Q)$ for the class of a
finite projective resolution of $M$. 
For an abelian group $A$, define $A_\Q := A \otimes_\Z \Q$. 

\begin{cor} \label{cor727} 
Let $p$ be a prime. Assume $Q$ is a regular $A_p$-algebra of dimension $d$ and $Z \subseteq \Spec(Q)$ is a closed subset of codimension $c$.  Then there is a
decomposition 
$$
K_0^Z(Q)_\Q = \bigoplus_{i=c}^d K_0^Z(Q)_\Q^{(i)}
$$
where $K_0^Z(Q)_\Q^{(i)}$ is the eigenspace of $\cPsi^p$ of eigenvalue $p^i$. Moreover, 
if $M$ is a finitely generated $Q$-module supported on $Z$, then 
$$
[M] \in \bigoplus_{i=\codim_Q(M)}^d K_0^Z(Q)^{(i)}_\Q.
$$ 
\end{cor}

\begin{proof} Gillet and Soul\'e \cite[5.3]{GS87} prove that  such a decomposition exists for any ``degree $k$ Adams operation'' $\psi^k$, satisfying their four
  axioms (A1)--(A4).
The result thus follows from 
Theorem \ref{psi-GS-axioms}. (The Gillet-Soul\'e axioms involve all schemes, and their Proposition 5.3 applies to all regular schemes, but their  proof of this result
applies to regular affine schemes if
the axioms hold just for affine schemes.)
\end{proof}

We can use our cyclic Adams operators to recover a proof of Serre's Vanishing Theorem. This was proven by Gillet and Soul\'e using the Adams operations they
construct. But 
notice that the construction of our operators $\cPsi^p$ does not involve the use of $\lambda$-operations 
nor the fact that $\bigoplus_Z K_0^Z(Q)$ is a special lambda ring, both of which
are complicated aspects of Gillet and Soul\'e's proof.

\begin{cor}[Serre's Vanishing Conjecture] \label{cor728c} 
(cf. \cite[Corollary 5.6]{GS87})
Let $(Q, \fm)$ be a regular local ring of dimension $d$ 
and let $M$ and $N$ be finitely generated $Q$-modules such that $\supp(M) \cap \supp(N) = \{\fm\}$. Define $\chi_Q(M,N) = \sum_i (-1)^i \len \Tor_i^Q(M,N)$. 
If $\codim_{\Spec Q}(\supp M) + \codim_{\Spec Q} (\supp N) > d$,  then $\chi_Q(M,N) = 0$. 
\end{cor}

\begin{proof} 
Let $p$ be any prime distinct from the residue characteristic of $Q$, so that $p$ is invertible in $Q$. 
We start by reducing to the case where $Q$ contains the $p$-th roots of unity. Since $\frac{1}{p} \in Q$, 
the map $Q \to Q[e^{\frac{2 \pi i}{p}}]$ is finite, \'etale. Let $Q'$ be the localization of 
$Q[e^{\frac{2 \pi i}{p}}]$ at any one of the maximal ideals lying over $\fm$. Then $\fm Q'$ is the maximal ideal of $Q'$ and 
$\Tor_j^Q(M,N) \otimes_Q Q' \cong 
\Tor_j^{Q'}(M \otimes_Q Q', N \otimes_Q Q')$ for all $j$, and hence
$\chi_Q(M,N) = \chi_{Q'}(M \otimes_Q Q', N \otimes_Q Q')$.

We may thus assume
that $Q$ contains all $p$-th roots of unity. In this case $\cPsi^p$ satisfies the Gillet-Soul\'e axioms by Theorem \ref{psi-GS-axioms},  
and hence the proofs of \cite[5.4, 5.6]{GS87} apply verbatim. 
\end{proof}

\section{Commutativity of the cyclic Adams operations} \label{seccommute}

In this section, we prove that the cyclic Adams operations commute, when defined.
For any integer $k \geq 1$, set 
$A_k =  \Z\left[\scriptstyle{\frac{1}{k}}, e^{\scriptscriptstyle{\frac{2 \pi i}{k}}}\right]$.

\begin{prop} 
Assume $p$ and $q$ are distinct primes and $Q$ is a commutative, Noetherian $A_{pq}$-algebra. 
Then for any closed subset $Z$ of $\Spec(Q)$, we have an equality
$$
\cPsi^p \circ \cPsi^q = \cPsi^q \circ \cPsi^p
$$
of endomorphisms of $K_0^Z(Q)$.
\end{prop}

\begin{proof}
Let $\zeta_p, \zeta_q$ be primitive $p$-th, $q$-th roots of unity, and consider the diagram
$$
\xymatrix{
K_0^Z(Q) \ar[r]^{t^p} \ar[rd]_{\cPsi^p} & K_0^Z(Q; C_p) \ar[d]^{\phi^p} \ar[r]^{t^q} & K_0^Z(Q; C_p \times C_q) \ar[d]^{\phi^p} \\
& K_0^Z(Q)[\zeta_p] \ar[rd]_{\cPsi^q} & K_0^Z(Q; C_q)[\zeta_p] \ar[d]^{\phi^q} \\
& & K_0^Z(Q)[\zeta_p, \zeta_q],
}
$$
where for $l = p,q$, $t^l$ is defined as the composition of the map $t_\Sigma^l$ of Theorem \ref{thm-t^k} with the restriction map induced by the inclusion
$Q[C_l] \subseteq  Q[\Sigma_l]$,
and $\phi^l$ is defined as in the proof of Theorem
\ref{psi-GS-axioms}. 
Since $\cPsi^q$ and $\phi^q$ are group homomorphisms, the two bottom arrows are well-defined.

The triangle in this diagram commutes by definition, and we will show the trapezoid commutes momentarily. Granting this, we obtain
\begin{equation} \label{E925}
\begin{aligned}
\cPsi^q (\cPsi^p([X])) 
&= \phi^q(\phi^p(t^q(t^p([X]))))\\
&=\sum_{i,j}\zeta_p^i\zeta_q^j[(T^{pq}(X)^{(\zeta_p^i)})^{(\zeta_q^j)}]\\
&= \sum_\eta \eta [T^{pq}(X)^{(\eta)}],
\end{aligned}
\end{equation}
where the last sum ranges over all $pq$-th roots of unity $\eta$. The last equality follows from the fact that for any complex with a $C_p\times C_q$ action,
the actions of $C_p$ and $C_q$ commute and each action is diagonalizable.  More precisely, for a complex $U$ with an action of $C_p\times C_q$, 
we claim that for each fixed $i$ and $j$ satisfying $1\leq i\leq p$ and $1\leq j\leq q$ we have
$$
\left(U^{(\zeta_p^i)}\right)^{(\zeta_q^j)}=U^{(\zeta_p^i\zeta_q^j)}.
$$
The containment $\subseteq$ is clear by definition, and as $U$ decomposes 
both as $\displaystyle U=\bigoplus_{j}\bigoplus_{i}\left(U^{(\zeta_p^i)}\right)^{(\zeta_q^j)}$ 
and as $\displaystyle U=\bigoplus_{i,j} U^{(\zeta_p^i\zeta_q^j)}$, equality follows.
Since the last expression of \eqref{E925} is symmetric in $p$ and $q$, the proposition follows.

It remains to show the trapezoid commutes. Let $Y \in \cP^Z(Q, C_p)$. Then $Y$ decomposes as $Y = \bigoplus_\zeta Y^{(\zeta)}$ where $\zeta$ ranges over all
$p$-th roots of unity. 
We have 
$$
\phi^q (\phi^p(t^q([Y])))  = \phi^q \left(\sum_{\zeta} \zeta [T^q(Y)^{(\zeta)}] \right)
= \sum_{\zeta} \zeta \phi^q ([T^q(Y)^{(\zeta)}])
$$
where the sum ranges over all $p$-th roots of unity and the superscript $(\zeta)$ refers to the (diagonal) $C_p$-action on $T^q(Y)$.  

For a fixed $p$-th root of unity $\zeta$, let $S_\zeta$ denote the set $\{(\zeta_1, \dots, \zeta_q)\}$ of ordered $q$-tuples of $p$-th roots of unity satisfying
$\zeta_1 \cdots \zeta_q = \zeta$. The group $C_q$ acts on $S_\zeta$ in the evident way. 
We have 
$$
T^q(Y)^{(\zeta)} = \bigoplus_{(\zeta_1, \dots, \zeta_q) \in S_\zeta} Y^{(\zeta_1)} \otimes_Q \cdots \otimes_Q Y^{(\zeta_q)}.
$$
as objects of $\cP^Z(Q, C_q)$, where the action of $C_q$ on the right is given by the action on $S_\zeta$. For each orbit $O \subseteq S_\zeta$,  
the summand 
$$
Y_O := \bigoplus_{(\zeta_1, \dots, \zeta_q) \in O} Y^{(\zeta_1)} \otimes_Q \cdots \otimes_Q Y^{(\zeta_q)}
$$
of $T^q(Y)^{(\zeta)}$ is an object of $\cP^Z(Q; C_q)$, and so we have 
$$
\phi^q ([T^q(Y)^{(\zeta)}])
= \sum_{O} \phi^q ([Y_O]).
$$
Since $q$ is prime, each orbit $O$ has order either $1$ or $q$. In the latter case 
$Y_O$ is extended, since
$$
Y_O \cong Q[C_q] \otimes _Q Y^{(\zeta_1)} \otimes_Q \cdots \otimes_Q Y^{(\zeta_q)},
$$
where $(\zeta_1, \dots, \zeta_q)$ is any chosen element of $O$, and hence, by Lemma~\ref{lem2}, $\phi^q ([Y_O]) = 0$ for such orbits. If $|O| = 1$, then its only element is $(\zeta_1, \dots, \zeta_1)$ with $\zeta_1^q = \zeta$. 
In this case, $Y_O  = T^q(Y^{(\zeta_1)})$ and so 
$$
\phi^q ([Y_O])  = \cPsi^q(Y^{(\zeta_1)}).
$$
Using that $(p,q) = 1$ and $Y^{(\zeta^q)} \cong Y^{(\zeta)}$ in $\cP^Z(Q)$, we conclude
$$
\phi^q (\phi^p(t^q([Y])))  = 
\sum_{\zeta_1} \zeta_1^q \cPsi^q([Y^{(\zeta_1^q)}]) =
\sum_{\zeta} \zeta \cPsi^q([Y^{(\zeta)}]) =
\cPsi^q (\phi^p([Y])).
$$
\end{proof}

Using this proposition, we extend the definition of the cyclic Adams operations to all positive integers.
If $k = p_1^{e_1} \cdots p_l^{e_l}$ is the prime factorization of an integer $k$ and $Q$ is an $A_k$-algebra,
we set
$$
\cPsi^k = (\cPsi^{p_1})^{\circ e_1} \circ \cdots \circ (\cPsi^{p_l})^{\circ e_l}. 
$$

\begin{rem} It seems likely that
$$
\cPsi^k([X]) = \sum_\zeta \zeta [T^k(X)^{(\zeta)}],
$$
where the sum ranges over all $k$-th roots of unity. This formula is known to hold in other contexts; see, e.g., Benson \cite{Ben84} and Theorem \ref{thm:Atiyah} below. 
\end{rem}

\section{Lambda operations and agreement with those of Gillet-Soul\'e}
\label{E^k}

As mentioned in the introduction, 
Gillet-Soul\'e \cite{GS87}  equip
$K_0^Z(Q)$ with $\lambda$ operations by using the Dold-Puppe construction of exterior powers on chain complexes.
The goal of this section is to prove that, for each $k \geq 1$, if $k!$ is invertible in $Q$, then the Gillet-Soul\'e operator $\lambda^k_{GS}$ agrees with the operator
given by taking ``naive'' $k$-th exterior powers of complexes.
We believe that this fact has been observed before by others, including M.\ Hashimoto (see \cite[\S2]{KuranoRoberts}),
but, as far as we know, a proof is not available in the literature. We therefore provide a careful one here.  

Let us explain what we mean by the ``naive'' exterior powers of a complex.
Let $Q_{\sign}$ denote $Q$ endowed with the structure of a left $Q[\Sigma_k]$-module via the sign representation: $\s \cdot q = \sign(\sigma) q$ for $\s \in
\Sigma_k, q \in Q$.
For a complex of $Q$-modules $Y$ equipped with an action of $\Sigma_k$,   define $Y^{(\sign)} = \Hom_{Q[\Sigma_k]}(Q_{\sign}, Y)$.
For any complex of $Q$-modules $X$, 
define 
$$
\Lambda^k_Q(X) := T^k(X)^{(\sign)}, 
$$
where $\Sigma_k$ acts on $T^k(X)$ as before.

For example, if $X$ is concentrated in even degrees, then
$\Lambda^k_Q(X)$ is the usual $k$-th exterior power of $X$,
realized as the submodule of anti-symmetric tensors in the $k$-th tensor power of $X$. 
Similarly if $X$ is concentrated in odd degrees, $\Lambda^k_Q(X)$ is the $k$-th
divided power of $X$.

\begin{rem} We define $\Lambda^k_Q(X)$ as a submodule of $T^k(X)$, but one could just as well define it to be a quotient module, using instead the formula
$X \otimes_{Q[\Sigma_n]} Q_\sign$. For $X$ concentrated in even degrees, this gives the usual $k$-th exterior power realized as a quotient of the $k$-th tensor
power. When $X$ is concentrated in odd degrees, one gets the $k$-th {\em symmetric} power, realized as a quotient in the standard way. If 
$k!$ is a unit in $Q$, these are naturally isomorphic constructions.
\end{rem}

The reason we call these ``naive'' exterior powers is that, in general, they do not preserve acyclicity, as the following example shows.

\begin{ex} Let $X$ be a complex of projective $Q$ modules concentrated in degrees $1$ and $2$:
$$
X = (\cdots \to 0 \to P_2 \xra{d} P_1 \to 0 \to \cdots).
$$
Then 
$$
\begin{aligned}
\Lambda_Q^k(X) = (\cdots & \to 0 \to \Lambda^k_Q(P_2) \to 
\Lambda^{k-1}_Q(P_2) \otimes_Q P_1 \to 
\Lambda^{k-2}_Q(P_2) \otimes_Q \Gamma_Q^2(P_1) \to \cdots \\
&  \to \Lambda^2_Q(P_2) \otimes_Q \Gamma_Q^{k-2}(P_1) \to 
P_2 \otimes_Q \Gamma_Q^{k-1}(P_1) \to \Gamma_Q^k(P_1) \to 0 \to \cdots),
\end{aligned}
$$
where $\Lambda^i_Q(P_2)$ and $\Gamma^j_Q(P_1)$ are  usual exterior and divided powers of (non-graded) $Q$-modules. In this complex $\Lambda^i_Q(P_2) \otimes_Q
\Gamma^{k-i}_Q(P_1)$ lies in degree $k +i$, and the differential sends $x_1 \smsh \cdots \smsh x_i \otimes \g$ to 
$$
\sum_j (-1)^{j-1} x_1 \smsh \cdots \smsh x_{j-1} \smsh x_{j+1} \smsh \cdots \smsh x_i \otimes d(x_j) \cdot \g
$$
where $\cdot$ is the multiplication operator for the divided power algebra $\Gamma_Q(P_1)$
(note that, since $\Lambda^i_Q(P_2)$ is by definition a submodule of $T^i_Q(P_2)$, the symbol $x_1 \smsh \cdots \smsh x_i$ should be interpreted as being the
element $\sum_{\s \in \Sigma_i} \sign(\s) x_{\s(1)} \otimes \cdots \otimes x_{\s(i)}$).

Now suppose that $P_1$ is free with basis $x_1, \dots, x_n$, $P_2 = P_1$, and $d$ is the identity map. Then $\Lambda^k(X)$ is in fact a summand of
the Koszul complex for the commutative ring $\Gamma_Q(P_1)$ on the sequence $x_1, \dots, x_n$ regarded as elements of $\Gamma_Q(P_1)$:
$$
\Lambda^k(X) = (\cdots \to \Gamma_Q^{k-2}(P_1)^{\oplus {n \choose 2}}  \to
\Gamma_Q^{k-1}(P_1)^{\oplus n} \xra{(x_1, \dots, x_n) }
\Gamma_Q^k(P_1) \to 0 \to \cdots) 
$$
Taking $n = 1$ and $x = x_1$ gives
$$
\begin{aligned}
\Lambda^k(X) & = (\cdots \to 0 \to Q x^{(k-1)} \xra{x} Q x^{(k)} \to 0 \to \cdots) \\
& \cong
(\cdots \to 0 \to Q \xra{k} Q \to 0 \to \cdots). \\
\end{aligned}
$$
So if $k$ is not invertible in $Q$, $\Lambda^k_Q(X)$ is not acyclic even though $X$ is.

Since the Dold-Puppe construction does preserve acyclicity, this example also shows that one must invert $k$ in order for the two $k$-th exterior power
operations to agree up to quasi-isomorphism.
\end{ex} 

\begin{rem} We will need to assume $k!$ is a unit, not just that $k$ is a unit,  in order to show that the naive $k$-th  exterior power agrees, up
  to quasi-isomorphism, with  
the $k$-th exterior power defined by  Dold-Puppe. We do not know if this assumption is essential (but suspect that it is).  
It turns out that assuming merely that $k$ is invertible suffices for the 
naive $k$-th exterior power to preserve acyclicity.
\end{rem}

\begin{rem} If $X$ is concentrated in degrees $1$ and $0$, then $\Lambda^k_Q(X)$ is quasi-isomorphic to the complex
$N(\Lambda^k(K(X))$ defined by Dold and Puppe, without any assumptions on $Q$; see \cite[2.7]{Koc01}.
\end{rem}

We now  assume $k!$ is a unit in $Q$.  In this case, (the proof of) Maschke's Theorem implies that 
$Q_{\sign}$ is a summand of $Q[\Sigma_p]$ and hence is a projective $Q[\Sigma_k]$-module.
The functor  $Y \mapsto Y^{(\sign)}$ is thus exact and
hence determines a homomorphism
$$
K_0^Z(Q; \Sigma_k) \to K_0^Z(Q).
$$
We define 
\begin{equation} \label{E1222}
\l^k_\naive: K_0^Z(Q) \to K_0^Z(Q)
\end{equation}
to be the composition of
$$
K_0^Z(Q) 
\xra{t^k_\Sigma} K_0^Z(Q, \Sigma_k) 
\xra{[Y] \mapsto [Y^{(\sign)}]}  K_0^Z(Q), 
$$
where the first map is from Theorem~\ref{thm-t^k}. So, 
$$
\l_\naive^k([X]) = [\Lambda_Q^k(X)]
$$
for any $X \in \cP^Z(Q)$.

\begin{rem} The previous example shows that one must, at the least, assume $k$ is a unit in order for $\l_\naive$ to be well-defined. We presume one must in fact
  assume $k!$ is invertible in order for it to be well-defined.
The issue is that $\Lambda^k_Q(-)$ seems unlikely to preserve
 all quasi-isomorphisms unless $k!$ is invertible.  
\end{rem}
 
The following result is likely well-known to the experts, but we include a formal proof for lack of a suitable reference. It applies to an arbitrary complex $M$
of $Q$-modules that
is concentrated in non-negative degrees. For such a complex $M$, $\Lambda^k_Q(M)$ is, as before, defined to be  $T^k_Q(M)^{(\sign)} = \Hom_{Q[\Sigma_k]}(Q^{(\sign)},
T^k_Q(M))$.

Likewise, for a simplicial $Q$-module $A$, we define $T^k_Q(A)$ to be the simplicial module obtained by applying the functor $T_Q^k(-)$ degreewise to $A$. 
The simplicial $Q$-module $T^k_Q(A)$ has an evident (unsigned) action of $\Sigma_k$, given by
$\sigma \cdot (a_1 \otimes \cdots \otimes
a_k) = a_{\sigma(1)} \otimes \cdots \otimes a_{\sigma(k)}$, making it a simplicial left $Q[\Sigma_k]$-module.
Finally, 
$\Lambda^k_Q(A)$ is defined to be the simplicial $Q$-module obtained by applying $\Hom_{Q[\Sigma_k]}(Q^{(\sign)}, -)$
degreewise  to $T^k_Q(A)$.

\begin{prop} \label{prop:folklore}
Let $Q$ be a commutative ring and $k \geq 1$ and integer such that $k!$ is invertible in $Q$. 
If $M$ is any complex of 
$Q$-modules concentrated in non-negative degrees, then there is a natural quasi-isomorphism
$$
\Lambda^k_Q(M) 
\xra{\sim}
N\!\left(\Lambda^k_Q(K(M))\right)
$$
of chain complexes, where  $K$ denotes the Dold-Puppe functor from
  chain complexes concentrated in non-negative degrees to simplicial modules and $N$ denotes 
the functor taking a simplicial $Q$-module to its associated normalized chain complex.
\end{prop}

\begin{proof}
We first recall some well-known results about simplicial modules 
and their normalized chain complexes. These ideas go back to Eilenberg and Maclane \cite[5.3]{EML53}; 
we refer the reader to the nice exposition found in \cite{SS03} for more details. 

For a simplicial $Q$-module $A$ and integer $k \geq 1$, 
we have the {\em shuffle map} 
$$
\nabla: T^k(N(A)) \to N(T^k(A)) 
$$
and the {\em Alexander-Whitney} map
$$
AW: N(T^k(A))  \to T^k(N(A)),
$$
both of which are natural transformations, and they satisfy the following properties:
\begin{itemize}
\item $AW \circ \nabla$ is the identity map,
\item $\nabla \circ AW$ is chain homotopic to the identity map, and
\item $\nabla$ is equivariant for the actions of $\Sigma_k$.
\end{itemize}
Concerning the last point, 
the action of $\Sigma_k$ on $T^k(N(A)) $ is the one we introduced in \S \ref{sectwo} above 
on a tensor power of a complex, and the action of $\Sigma_k$ on
$N(T^k(A))$ is induced from the action on the simplicial module $T^k(A)$, 
using that $N$ is functorial. Indeed, since $\nabla$ is injective, the action of 
$\Sigma_k$ on $T^k(N(A))$ is the unique one making $\nabla$ equivariant. 

It is important to note that $AW$ is {\em not} equivariant for the action of $\Sigma_k$, and so while
$\nabla$ is an equivariant quasi-isomorphism, it is not in general
an equivariant homotopy equivalence. 
This defect is precisely why we must assume $k!$ is invertible for this proof. 

Since we are assuming $k!$ is invertible, $Q^{(\sign)}$ is a projective  $Q[\Sigma_k]$-module and thus
the functor 
$(-)^{(\sign)} = \Hom_{Q[\Sigma_k]}(Q^{(\sign)}, -)$ takes 
equivariant quasi-isomorphisms (i.e., quasi-isomorphisms of complexes of left $Q[\Sigma_k]$-modules) to quasi-isomorphisms. In particular,  the map
$$
\nabla^{(\sign)}: 
\Lambda_Q^k(N(A)) = T^k(N(A))^{(\sign)}
\xra{\sim}
N(T^k(A))^{(\sign)}
$$
induced by the shuffle map is a quasi-isomorphism. 
Moreover, by Lemma \ref{lemma-hom-N} below 
applied with $R = Q[\Sigma_k]$, $V = Q^{(\sign)}$, and $B = T^k(A)$, 
we have an isomorphism
\begin{align*}
N(T^k(A))^{(\sign)} 
&=
\Hom_{Q[\Sigma_k]}(Q^{(\sign)}, N(T^k(A)))  \\
&\cong 
N\!\left(\Hom_{Q[\Sigma_k]}(Q^{(\sign)},T^k(A) )\right) = N\!\left(\Lambda^k_Q(A) \right). 
\end{align*}
Putting these results together gives a natural quasi-isomorphism of complexes of $Q$-modules
\begin{equation} \label{E1223}
\Lambda^k_Q(N(A)) 
\xra{\sim}
N\!\left(\Lambda^k_Q(A) \right)
\end{equation} 
for any simplicial $Q$-module $A$.

Now let $M$ be a complex of $Q$ modules concentrated in non-negative degrees. 
Since there is a natural isomorphism $N(K(M)) \cong M$ of complexes,  
we have $T^k(N(K(M)))^{(\sign)} \cong T^k(M)^{(\sign)} = \Lambda^k_Q(M)$, and thus the proposition follows from 
\eqref{E1223} applied to $A = K(M)$. 
\end{proof}

\begin{lem} \label{lemma-hom-N} 
Let $R$ be a (not necessarily commutative) ring and $B$ a simplicial left $R$-module. 
Let $V$ be any left $R$-module. Then
$$
N(\Hom_R(V, B)) \cong  \Hom_R(V, N(B))
$$
where $N(-)$ is the functor taking simplicial modules to normalized chain complexes.
\end{lem}

Note that in this lemma, $\Hom_R(V,B)$ denotes applying $\Hom_R(V, -)$ 
degreewise to $B$, and so $\Hom_R(V,B)$ is a simplicial module over the center of 
$R$. On the other hand, $\Hom_R(V, N(B))$ denotes applying $\Hom_R(V, -)$ 
degreewise to a chain complex in the usual way. The isomorphism is an isomorphism 
of chain complexes over the center of $R$.

\begin{proof} 
Let $C$ denote the functor sending simplicial modules to non-normalized complexes. 
Recall $C(B)_n = B_n$ with differential given by the alternating sum of the
boundary maps. Then it is clear that we have an equality (not just an isomorphism)
$$
C(\Hom_R(V,B)) = \Hom_R(V, C(B)).
$$
Now $N(B)$ is by definition the quotient complex of $C(B)$ defined by modding out 
by the subcomplex generated by degenerate simplices, but it is easier to prove the lemma by using the Moore complex:
this is defined as the subcomplex $M(B)$ of $C(B)$ with
$$
M(B)_n = \cap_{i=1}^n  \ker(d_i: B_n \lra B_{n-1}).
$$
Note that the differential on $M(B)$ is $d_0$. 

The composition of 
$$
M(B) \lra C(B) \lra N(B) 
$$
is an isomorphism of chain complexes, and the lemma thus follows from the equality
$$
M(\Hom_R(V, B)) = \Hom_R(V, M(B)),
$$
which is evident from the definitions.
\end{proof}

\begin{thm} \label{gs=hashimoto} 
If $Q$ is a commutative, Noetherian ring and $k \geq 1$ is an integer such that $k!$ 
is invertible in $Q$, then the functions $\l_\naive^k$ and $\lambda^k_{GS}$ on $K_0^Z(Q)$ coincide for all $Z$.
\end{thm}

\begin{proof} Proposition \ref{prop:folklore} gives that
$\l_\naive^k([P]) = \l_{GS}^k([P])$ if $P \in \cP^Z(Q)$ is concentrated in non-negative degrees.
To conclude $\l_\naive^k$ and  $\l^k_{GS}$ coincide on all elements of $K_0^Z(Q)$, 
recall that every element of $K_0^Z(Q)$ equals $[P]$ for some complex $P \in \cP^Z(Q)$, 
and moreover we have $[P] = [P[2n]]$ for all $n \in \Z$ (where
$P[i]$ is the complex with $P[i]_d = P_{i+d}$). For $n$ sufficiently negative, 
$P[2n]$ is concentrated in non-negative degrees.
\end{proof}

\section{Agreement of $\cPsi$ with $\psi_{GS}$}
\label{agreement}

As in the previous section, we assume $Q$ is a commutative, Noetherian ring and $k$ is a positive integer such that $k!$ is invertible in $Q$. By Theorem
\ref{gs=hashimoto},
$\l^k_\naive = \l^k_{GS}$ on $K_0^Z(Q)$ for all $Z$ and thus in this section we write this common operator simply as $\l^k$.
The main result of this section, which is ultimately due to Atiyah, is that the $k$-th Adams operation defined via the $k$-th Newton polynomial in the $\l^i$
operators, $1\leq i\leq k$, agrees with the cyclic Adams
operation $\cPsi^k$.

\def\Zk{{\Z[{\scriptsize{\frac{1}{k!}}}]}}

In \cite{Ati66}
Atiyah  works over the complex numbers, and so some care is needed to adapt his argument to our situation. 
For a commutative ring $A$ and finite group $G$, let $P(A; G)$ denote the category of left $A[G]$-modules that are finitely generated and projective as $A$-modules,
and write $R_A(G)$ for its Grothendieck group. $R_A(G)$ is naturally isomorphic to the group $K_0^{\Spec(A)}(A; G)$ introduced above. 
We will only consider $R_A(G)$ when  $|G|$ is invertible in
$A$, and in this case (the proof of) Maschke's Theorem gives that a left $A[G]$-module is projective as an $A[G]$-module if and only if it is projective as an $A$-module. 

For groups $G, G'$ there is a pairing
$$
R_A(G) \otimes_\Z R_A(G') \to R_A(G \times G')
$$
induced by $- \otimes_A -$ with the evident $G \times G'$-action. Taking $G = \Sigma_i$ and  $G' = \Sigma_j$ and composing this pairing with extension of scalars along $\Sigma_i
\times \Sigma_j \cong \Sigma_{i,j} \leq \Sigma_{i+j}$ induces a pairing
\begin{equation} \label{E84}
- \star - : R_A(\Sigma_i) \otimes_\Z R_A(\Sigma_j) \to R_A(\Sigma_{i+j}),
\end{equation}
which is a special case of the pairing constructed in \S \ref{sectwo}.
As before, this pairing is commutative and associative and hence determines a non-unital ring
$$
R_A(\Sigma) = \bigoplus_{i \geq 1} R_A(\Sigma_i).
$$
(We could include $i = 0$, interpreting $R_A(\Sigma_0)$ as $\Z$, and make this into a unital ring, but for our purposes that would be less convenient.)
The dual version of this pairing is used by Atiyah in \cite[\S 1]{Ati66} and the version that occurs here (defined only over the complex numbers) 
appears in work of Knutson, where he calls it the ``outer product'' \cite[Page 127]{Knu73}.

For a closed subset $Z$ of $\Spec(Q)$ define $\Op K_0^Z(Q)$ to be the set of functions from $K_0^Z(Q)$ to itself. Since
$K_0^Z(Q)$ is a non-unital ring under $\cup$, $\Op K_0^Z(Q)$ is also 
a non-unital ring with  $(\a + \b)(x) = \a(x) + \b(x)$ and $(\a \cup \b)(x) = \a(x) \cup \b(x)$, for
$\a, \b \in \Op K_0^Z(Q)$ and
 $x \in K_0^Z(Q)$ (note that composition of operations is not involved).

\begin{rem} It is perhaps more sensible to define $\Op K_0^Z(Q)$ to consist of just the operations that are natural for ring
  homomorphisms in a suitable sense, as Atiyah does in a different context. The larger ring used here turns out to be more useful for our purposes.
\end{rem}

Recall we assume $Q$ is a $\Zk$-algebra. Then for $1 \leq i \leq k$, we have a bilinear pairing
$$
R_\Zk(\Sigma_i) \otimes_\Z K_0^Z(Q; \Sigma_i) \to K_0^Z(Q)
$$
given by $[M] \otimes [Y] \mapsto \Hom_{\Zk[\Sigma_i]}(M, Y)$ for $M \in P(\Zk; \Sigma_i)$ and $Y \in \cP^Z(Q; \Sigma_i)$.   
This is well-defined since the assumption that $k!$ is invertible
ensures that $M$ is projective as a left $\Zk[\Sigma_i]$-module.

Precomposing this pairing with the map $t^i_\Sigma: K_0^Z(Q) \to K_0^Z(Q; \Sigma_i)$ defined in Theorem \ref{thm-t^k} gives an additive map
$$
t: R_\Zk(\Sigma_i) \to \Op K_0^Z(Q),
$$
written $\rho \mapsto t_\rho$. 
In detail, for $\rho = [M]$ with $M \in P(\Zk; \Sigma_i)$, we have 
$$
t_\rho([X]) = [\Hom_{\Zk[\Sigma_i]}(M, T^k(X))] \in K_0^Z(Q),
$$
for any  $X \in \cP^Z(Q)$. 

\begin{ex} \label{ex86} If $\s_i \in R_\Zk(\Sigma_i)$ is the class of the sign representation, then $t_{\s_i} = \l_\naive^i$, where $\l_\naive^i$ is defined in
  \eqref{E1222}.
\end{ex}

\begin{lem} \label{lem84}
If $i + j \leq k$, then for $\rho \in R_\Zk(\Sigma_i), \rho' \in R_\Zk(\Sigma_j)$ we have $t_\rho \cup t_{\rho'} = t_{\rho \star \rho'}$.
\end{lem}

\begin{proof} 
Since $t$ is additive, we may assume $\rho$ and $\rho'$ are classes represented by modules $M \in
P(\Zk; \Sigma_i)$, $N \in P(\Zk; \Sigma_j)$. 
For any $X$ in $\cP^Z(Q)$ we have the $Q$-linear isomorphisms
$$
\begin{aligned}
\Hom_{\Zk[\Sigma_i]} & (M, T^i(X))  \otimes_Q \Hom_{\Zk[\Sigma_j]}(N, T^j(X)) \\
& \cong  \Hom_{\Zk[\Sigma_i] \otimes_\Zk \Zk[\Sigma_j]} (M \otimes_{\Zk} N, T^i(X)  \otimes_Q T^j(X)) \\
& = \Hom_{\Zk[\Sigma_{i,j}] } (M \otimes_{\Zk} N, T^{i+j}(X)) \\
& \cong \Hom_{\Zk[\Sigma_{i+j}] } \left(  (M \otimes_{\Zk} N)   \otimes_{\Zk[\Sigma_{i,j}] } \Zk[\Sigma_{i+j}], T^{i+j}(X)\right), \\
\end{aligned}
$$
where the last isomorphism holds by adjointness since  $T^{i+j}(X)$ is a $\Zk[\Sigma_{i+j}]$-module. 
The result follows.
\end{proof}

We now relate these constructions to Atiyah's work.

When $A = \C$, instead of the ring $R_\C(\Sigma)$ considered here, Atiyah defines a non-unital ring $R_\C(\Sigma)^* = \bigoplus_{i \geq 1} R_\C(\Sigma_i)^*$, where
$R_\C(\Sigma_i)^* := \Hom_\Z(R_\C(\Sigma_i), \Z)$. The multiplication rule is given by the composition of
$$
R_\C(\Sigma_i)^* \times R_\C(\Sigma_j)^* \cong R_\C(\Sigma_i \times \Sigma_j)^* \xra{\res^*} R_\C(\Sigma_{i +j})^*
$$
where $\res$ is restriction of scalars and $\res^*$ is its $\Z$-linear dual.

In fact, the rings $R_\C(\Sigma)^*$ and $R_\C(\Sigma)$ are canonically isomorphic. For recall that for any finite group $G$ 
there is a $\Z$-linear perfect pairing
$$
\langle -,- \rangle: R_\C(G) \times R_\C(G) \to \Z
$$
sending $(\rho, \rho')$ to $\frac{1}{|G|} \sum_g \trace \rho(g) \trace \rho'(g^{-1})$. In particular, these pairings determine isomorphisms $R_\C(\Sigma_i)
\cong R_\C(\Sigma_i)^*$, which, by Frobenius reciprocity, are compatible with the multiplication maps, so that $R_\C(\Sigma) \cong R_\C(\Sigma)^*$ as rings.

One advantage the ring $R_\C(\Sigma)^*$ has is that there is an evident map
$$
C(\Sigma_k)^* \to R_\C(\Sigma_k)^*
$$
where $C(G)^*$ denotes the free abelian group on the set of conjugacy classes of a finite group $G$. (We may identify  $C(G)^*$ as the $\Z$-linear dual of $C(G)$, the set of $\Z$-valued
functions on $G$ that are invariant on conjugacy classes.) The map sends $\tau \in \Sigma_k$ to the function $[\rho] \mapsto \trace \rho(\tau)$ (recall that the
character table of $\Sigma_k$ has only integer entries). These maps assemble to form a ring homomorphism 
$$
C(\Sigma)^* = \bigoplus_{i \geq 1} C(\Sigma_i)^* \to R_\C(\Sigma)^*
$$
where, for $\tau \in \Sigma_i, \tau' \in \Sigma_j$, the multiplication rule on $C(\Sigma)^*$ sends $[\tau], [\tau']$ to the class of the element of $\Sigma_{i+j}$ given by the
inclusion $\Sigma_i \times \Sigma_j \cong \Sigma_{i,j} \leq \Sigma_{i+j}$. 

We thus also have a ring homomorphism
\begin{equation} \label{E1222b}
C(\Sigma)^* \to R_\C(\Sigma)
\end{equation}
sending the class of $\tau \in \Sigma_k$ to $\sum_i n_i \rho_i$, where the $\rho_i$'s range over representatives of the irreducible representations and $n_i = \trace \rho_i(\tau)$. 

\begin{defn} Define $\a^k \in R_\C(\Sigma)$ to be the image of the class of a $k$-cycle in $C(\Sigma_k)^*$ under the map \eqref{E1222b}. 
\end{defn}

As we shall see, $\a^k$ induces the $k$-th Adams
  operation, and the key point in proving this is given by the following result of Atiyah:

\begin{thm} \cite[1.8]{Ati66}  \label{thmAtiyah} In $R_\C(\Sigma)$ we have 
$$
\a^k = Q_k(\s_1, \dots, \s_k)
$$
where $Q_k$ is the $k$-th Newton polynomial and the $\s_i$'s are as in Example \ref{ex86}.
\end{thm}

To apply Atiyah's Theorem to our situation, we use:

\begin{lem} \label{lem81}
For any positive integer $k$, the map induced by extension of scalars
$$
R_\Zk(\Sigma_i) \to R_\C(\Sigma_i)
$$
is an isomorphism for $1 \leq i \leq k$. If $i + j \leq k$, then these isomorphisms commute with the pairings $- \star -$ of \eqref{E84}.
\end{lem}

\begin{proof} 
As is well-known, the map $R_\Q(\Sigma_i) \to R_\C(\Sigma_i)$ is
an isomorphism \cite[p. 37]{JamesKerber}, and so it suffices to prove $R_\Zk(\Sigma_i) \to R_\Q(\Sigma_i)$ is an isomorphism. 

For any regular ring $A$ in which $i!$ is invertible, we may identify $R_A(\Sigma_i)$ with the Grothendieck group of all finitely generated left
$A[\Sigma_i]$-modules. It thus follows from 
a theorem of Swan \cite[Theorem 2]{Swan} that we have a right exact sequence
\begin{equation} \label{E1-5}
\bigoplus_{p > k} R_{\Z/p}(\Sigma_i) \to R_\Zk(\Sigma_i)\to R_\Q(\Sigma_i) \to 0
\end{equation}
where the direct sum ranges over primes larger than $k$. Here, for each such prime $p$, the map 
$R_{\Z/p}(\Sigma_i) \to R_\Zk(\Sigma_i)$ is induced by restriction of scalars along the surjection $\Zk[\Sigma_i] \onto \Z/p[\Sigma_i]$. 
Every irreducible $\Z/p$-representation of $\Sigma_i$ lifts to the integers \cite[p. 244]{JamesKerber}, and thus  $R_{\Z/p}(\Sigma_i)$ is generated by classes of the form $[F \otimes_\Z
\Z/p]$ where $F$ is a $\Z[\Sigma_i]$-module that is finitely generated and free as a $\Z$-module. 
The image of such a class in $R_\Zk(\Sigma_i)$ is trivial since we have the short exact sequence
$$
0 \to F \otimes \Zk \xra{p} F \otimes \Zk \to F \otimes \Z/p \to 0
$$
of $\Zk[\Sigma_i]$-modules. Thus the first map in \eqref{E1-5} is the zero map.

This establishes the first assertion, and the second one is evident from the definitions. 
\end{proof}

\begin{defn} 
If $k!$ is invertible in $Q$, write $\a^k$ also for the element of $R_\Zk(\Sigma_k)$ corresponding to $\a^k \in R_\C(\Sigma_k)$ under the isomorphism of
Lemma \ref{lem81}. 

Define $\psi^k_{At} \in \Op K_0^Z(Q)$ to be $t_{\a^k}$.
\end{defn}

\begin{prop} If $k!$ is invertible in $Q$, then 
$$
\psi^k_{At} = Q_k(\l^1, \dots, \l^k) 
$$
holds in $\Op K_0^Z(Q)$.
\end{prop}

\begin{proof} This is an immediate consequence of Theorem \ref{thmAtiyah} and Lemmas \ref{lem84} and \ref{lem81}.
\end{proof}

\begin{thm} \label{thm:Atiyah}
If $Q$ is a commutative, Noetherian ring and $k$ is a positive integer such that $k!$ is invertible in $Q$ and $Q$ contains all the $k$-th roots of unity, then the operator
$$
\psi^k_{At}: K_0^Z(Q) \to K_0^Z(Q)
$$
is given by
$$
[X] \mapsto \sum_{\zeta} \zeta [T^k(X)^{(\zeta)}] 
$$
where the sum ranges over all $k$-th roots of unity $\zeta$ and $T^k(X)^{(\zeta)} := \ker(\s - \zeta \colon T^k(X) \to T^k(X))$ where $\s = (1 \, 2 \, \cdots \,
k)$. 
\end{thm}

\begin{rem} If $\zeta_1$ is a primitive $k$-th root of unity, then 
$$
\sum_{\zeta} \zeta [T^k(X)^{(\zeta)}] = \sum_{d \mid k} \mu(d) [T^k(X)^{(\zeta_1^{k/d})}]
$$ 
where $\mu$ is the M\"obius function. 
\end{rem}

\begin{proof} A version of this Theorem is proven by Atiyah in a different context, at least when $k$ is prime (see \cite[2.7]{Ati66}). We give a direct proof. 

Let $B = \Z[\frac{1}{k!}, \z]$ where $\z := e^{2 \pi i/k}$. Our assumptions on $Q$ make it a $B$-algebra. 

As we noted above, there are pairings 
$$
R_B(G) \otimes_\Z K_0^Z(Q; G) \to K_0^Z(Q)
$$
for any finite group $G$, and these extend to pairings
$$
R_B(G)[\z] \otimes_\Z K_0^Z(Q; G) \to K_0^Z(Q)[\z].
$$
For $\b \in R_B(G)[\z]$, let $\b_* \colon  K_0^Z(Q; G) \to K_0^Z(Q)[\z]$ denote the induced map.

The map 
$$
[X] \mapsto \sum_{j=0}^{k-1} \z^j [T^k(X)^{(\z^j)}] 
$$
is the composition of
$$
K_0^Z(Q) \xra{t^k_\Sigma} K_0^Z(Q; \Sigma_k) \xra{\res} K_0^Z(Q; C_k) \xra{\b^k_*} K_0^Z(Q)[\z]
$$
where $\b^k \in R(C_k)[\z]$ is $\sum_{j=0}^{k-1} \z^i [B_{\zeta^j}]$. (Here $B_{\z^j}$ is the rank one representation of $C_k$ over $B$ such that $(1 \, 2 \, \cdots
\, k)$ acts as multiplication by $\z^j$.) By Frobenius reciprocity, this coincides with 
the composition of
$$
K_0^Z(Q) \xra{t^k_\Sigma} K_0^Z(Q; \Sigma_k) \xra{\ind(\b^k)_*} K_0^Z(Q)[\z]
$$
where $\ind \colon R(C_k)[\z] \to R(\Sigma_k)[\z]$ is given by extension of scalars. It therefore suffices to prove $\ind(\b^k) = \a^k$ in $R(\Sigma_k)[\z] \supseteq
R(\Sigma_k)$. 

To prove this, since 
$$
\langle -,- \rangle \colon R_B(\Sigma_k)[\z] \otimes_{\Z[\z]} R_B(\Sigma_k)[\z] \to \Z[\z].
$$
is  a perfect pairing, it suffices to prove $\langle \a^k, - \rangle = \langle \ind(\b^k), - \rangle$. 
The map $\langle \a^k, - \rangle$ sends $\rho$ to $\trace(\rho(\s))$ where $\s := (1 \, 2 \, \dots \, k)$. On the other hand,
using Frobenius reciprocity for characters, $\langle \ind(\b^k), - \rangle = \langle \b^k, \res(-) \rangle$, where the pairing on the right is the one for the
group $C_k$, and so  
$$
\langle \ind(\b^k), \rho \rangle
=
\frac{1}{k}  \sum_{i=0}^{k-1} \trace(\b^k(\s^i)) \trace(\rho(\s^{-i})).
$$
Now, 
$$
\trace(\b^k(\s^i)) = \sum_{j=0}^{k-1} \z^j \z^{ij} =  \sum_{j=0}^{k-1} \z^{(i+1)j}
= \begin{cases}
k, & \text{if $i = k-1$ and} \\
0, & \text{if $0 \leq i \leq k-2$.} \\
\end{cases}
$$
We conclude that 
$$
\langle \ind(\b^k), \rho \rangle = 
\frac{1}{k} k \trace(\rho(\s^{-(k-1)})) = \trace(\rho(\s)).
$$
\end{proof} 

Recall that Gillet and Soul\'e define their $k$-th Adams operation by 
$$
\psi^k_{GS} = Q_k(\l^1, \dots, \l^k). 
$$

\begin{cor} \label{cor930} Let $Q$ be a commutative, Noetherian ring and 
$k$ an integer  such that $k!$ is invertible in $Q$ and $Q$ contains all the $k$-th roots of unity. Then
$\cPsi^k = \psi^k_{GS}$ as operators on  
$K_0^Z(Q)$.
\end{cor}

\begin{proof} Let $k = p_1^{e_1} \cdots p_m^{e_m}$ be the prime factorization of $k$. Recall that 
$$
\cPsi^k = \left(\cPsi^{p_1}\right)^{\circ e_1} \circ \cdots \circ \left(\cPsi^{p_m}\right)^{\circ e_m} 
$$
by definition. The analogous formula holds for $\psi^k_{GS}$ by \cite[I.6.1]{FL85}. It therefore suffices to consider the case that $k = p$ is prime. 
In this case, the result follows from Theorem \ref{thm:Atiyah}. 
\end{proof}


\bibliographystyle{amsalpha}
\bibliography{bibliography}

\end{document}